\documentclass[11pt]{article}

\usepackage[a4paper,margin=1in]{geometry}

\usepackage{amsmath,color}
\usepackage{amssymb}
\usepackage{amsthm}
\usepackage{mathtools}

\usepackage[T1]{fontenc}
\usepackage{lmodern}
\usepackage{microtype}

\usepackage{enumitem}
\usepackage{graphicx}
\usepackage{booktabs}
\usepackage{bbm}

\usepackage[hidelinks]{hyperref}
\usepackage[nameinlink,capitalize,noabbrev]{cleveref}

\newtheorem{theorem}{Theorem}[section]

\newtheorem{lemma}[theorem]{Lemma}
\newtheorem{proposition}[theorem]{Proposition}

\theoremstyle{definition}

\theoremstyle{remark}

\newcommand{\eps}{\varepsilon}

\numberwithin{equation}{section}

\newcommand{\farc}{\frac}
\newcommand{\be}{\begin{equation}}  
\newcommand{\ee}{\end{equation}}  

\newcommand{\bal}{\begin{aligned}}  
\newcommand{\enbal}{\end{aligned}}

\newcommand{\Rm}{{\mathbb R}}
\newcommand{\cT}{{\cal T}}

\newcommand{\cF}{{\cal F}}
\newcommand{\cS}{{\cal S}}

\newcommand{\tW}{\widetilde W}
\DeclareMathOperator{\pbmo}{p-BMO}
\newcommand{\one}{{\mathbbm{1}}}

\title{On the uniform bound of solutions to a thermo-diffusive system}
\author{Kirill Leontev\footnote{Department of Mathematics, Stanford University,
Stanford, CA 94305, USA; leontev@stanford.edu} 
\and Lenya Ryzhik\footnote{Department of Mathematics, Stanford University,
Stanford, CA 94305, USA; ryzhik@stanford.edu}}
\date{}
\begin{document}

\maketitle

\begin{abstract}
We prove a global-in-time uniform bound for solutions to thermo-diffusive systems
of reaction-diffusion equations with $L^\infty$ initial conditions.  
\end{abstract}

\section{Introduction}

We consider non-negative solutions to a thermo-diffusive system of the form
\begin{equation}\label{main-PDE}
\bal
u_t=\nu\Delta u-u g(v),\\
v_t=\kappa\Delta v+u g(v),
\enbal
\end{equation}
in the whole space $x\in\Rm^d$. In this simple model, $u(t,x)$ represents
the fuel concentration and~$v(t,x)$ is the fuel temperature. 
We assume that initial conditions are bounded: 
there exists~$K\ge 1$ so that~$u_0(x)=u(0,x)$ and $v_0(x)=v(0,x)$ satisfy 
\begin{equation}\label{main-PDE-initial-conditions}
0\leq u_0(x)\leq K,
\qquad
0\leq v_0(x)\leq K.
\end{equation}

For the nonlinearity, we assume that the reaction \(g\in C^1([0,\infty))\) satisfies
\begin{equation}\label{26jul2602}
g(0)=0,
\end{equation}
and that there exists $\rho\in(0,1)$ so that 
\begin{equation}\label{26aug1702}
0\le g(v)\leq c_2e^{Zv^\rho},~~\hbox{ for all $v>0$,}
\end{equation}
and, finally, that
\begin{equation}\label{g-liminf-bound}
    \liminf_{v \rightarrow \infty}g(v) > 0.
\end{equation}
Under these assumptions, the comparison principle implies immediately that 
\begin{equation}
u(t,x)>0,~~v(t,x)>0,~~\hbox{ for all $t>0$ and $x\in\Rm^d$.}
\end{equation}
This, in turn, implies that 
\begin{equation}\label{26jul2616}
u(t,x)\le \sup_{x\in\Rm^d}u_0(x)\le K,
\end{equation}
because of the upper bound on $u_0(x)$ in \eqref{main-PDE-initial-conditions}
and non-negativity of $g(v)$ for $v\ge 0$. 
The main result of this paper is the following uniform bound on the fuel temperature $v(t,x)$.
\begin{theorem}\label{thm-26aug902}
Let $u(t, x)$ and $v(t, x)$ be  solutions to \eqref{main-PDE}
under the assumptions \eqref{main-PDE-initial-conditions}-\eqref{g-liminf-bound} on initial conditions $u_0$ and $v_0$
and the nonlinearity $g(s)$. 
Then there exists $C_0>0$ so that 
\begin{equation}\label{26jul2604}
0< v(t,x)\leq C_0,~~\hbox{ for all $x\in\Rm^d$ and $t>0$.} 
\end{equation}
\end{theorem}
Let us recall  
that solutions to a single nonlinear heat equation
\[
\partial_tv=\Delta v+v^p,
\]
which corresponds to formally setting $u\equiv 1$ in the second equation in~\eqref{main-PDE},
may blow up in finite time~\cite{F1966}. Therefore, any uniform bound on the fuel temperature
must involve a non-trivial  interaction between the two equations in~\eqref{main-PDE}.
The simplest example of such interaction is that when $\kappa=\nu$,
solutions to~\eqref{main-PDE} are clearly uniformly bounded because 
the function~$z(t,x)=u(t,x)+v(t,x)$ solves the linear heat equation
\[
\partial_tz=\kappa\Delta z,
\]
so that both $u(t,x)$ and $v(t,x)$ are globally bounded in time.  
In addition, when $\kappa\le\nu$, the so-called Martin-Pierre trick~\cite{MartinPierre}
immediately shows that $v(t,x)$ is also uniformly bounded in time. We discuss it in 
Section~\ref{sec:martin-pierre} below. 
On the other hand, boundedness of the  temperature~$v(t,x)$ is a surprisingly  much more
difficult question when $\kappa\ge\nu$ and requires a much finer understanding of the interactions between
the two equations in~\eqref{main-PDE}. 

We refer to the introduction to~\cite{LaRoqRyz} for a recent review
of the existing literature on the problem, especially for the  
many surprises of the behavior of the solutions to reaction-diffusion systems in bounded domains, of which~\cite{P2010} gives a fascinating overview.
We should also mention the recent papers on reaction-diffusion equations
with reversible reactions: see~\cite{DFT2017,FLT2018} and references therein. Systems of thermo-diffusive type but with an additional heat loss were studied
in~\cite{BHKR2005,BN1992,Ducrot2021,HR2005,HR2010}, but all of these papers 
used the heat loss in essential ways.

Let us now describe a few of the existing results on the uniform bounds for the solutions to  the thermo-diffusive system~\eqref{main-PDE}
in the whole space for bounded initial conditions, without 
any decay at infinity of the initial fuel concentration $u_0$, in the ``difficult'' case $\kappa\ge\nu$.
To the best of our knowledge, the first non-trivial bound was obtained
by Collet and Xin in~\cite{CX1996}, where it was shown that $v(t,x)$ satisfies an upper bound
\begin{equation}
v(t,x)\le C_0\log\log t,~~\hbox{ for $t\gg 1$,}
\end{equation}
for polynomially growing nonlinearities $g(v)$. For an exponentially growing 
nonlinearity $g(v)$, global-in-time
existence of the solutions to~\eqref{main-PDE} was proved in~\cite{HLV} by  Herrero,  Lacey, and  Velazquez.
While no explicit
bounds on the solutions are stated in that paper, the  methods of~\cite{HLV}
are believed to yield an a priori bound of the type
\begin{equation}\label{26aug1310}
v(t,x)\le C_0\log t,~~\hbox{ for $t\gg 1$.}
\end{equation}

The one dimension Fisher-KPP thermo-diffusive system with $g(v)=v$:
\begin{equation}  
\begin{split}
&\partial_t u = \nu \partial_x^2u - uv,\\
&\partial_t v = \kappa \partial_x^2v + uv,
\end{split}
\end{equation}
was studied by Chen and Qin~\cite{CQ2008}, with the initial condition 
\begin{equation}\label{26aug1312} 
u_0(x)\equiv 1.
\end{equation} 
They 
have shown that the solution is uniformly bounded in time:
\begin{equation}\label{26aug1402}
v(t,x)\le C_0,~~\hbox{ for all $t\ge 0$.}
\end{equation} 
Their proof, while quite ingenious, relies heavily on the assumption that $g(v)=v$, as well as on the precise form~\eqref{26aug1312} of $u_0(x)$ and on the asymptotics 
for the location of the front of the solution to the single Fisher-KPP equation
originally proved in~\cite{Bramson1,Bramson2} by Bramson using probabilistic tools. 
This makes it very specific to $g(v)$ that arise from the McKean representation 
of the maximum of a branching
Brownian motion~\cite{McKean}. In particular, they must satisfy
the Fisher-KPP assumption: 
\begin{equation}
0<g(v)<g'(0)v,~~\hbox{for all $v>0$.}
\end{equation}
Most recently, it was proved  in~\cite{LaRoqRyz}, without the Fisher-KPP assumption, 
that the bound~\eqref{26aug1402} holds in any dimension $d\ge 1$ if the nonlinearity $g(v)$
satisfies a weaker condition 
\begin{equation}\label{26aug1316} 
\hbox{$g(v)>0$ for all $v>0$ and $g'(0)>0$,}
\end{equation}
and the initial condition also satisfies a weaker condition than~\eqref{26aug1312}: there exists $K\ge 1$ so that, in addition to~\eqref{main-PDE-initial-conditions}, we have 
\begin{equation}\label{26aug1314} 
u_0(x)+v_0(x)\ge K^{-1},~~\hbox{for all $x\in\Rm^d$.}
\end{equation}
Both of these assumptions have been used in~\cite{LaRoqRyz} to ensure the hair trigger effect: as soon as  
the fuel temperature~$v(t,x)$ becomes ``not too tiny'' in a given region,
it starts growing exponentially fast, which, in turn, causes the fuel concentration to drop exponentially 
fast in time. 
 
An important class of nonlinearities that do not satisfy~\eqref{26aug1316} is the class of ignition
nonlinearities that appear in combustion theory. They have the property that, in addition
to~\eqref{26aug1702}-\eqref{g-liminf-bound}, there exists $\theta>0$ (called the ignition temperature)
so that 
\begin{equation}
g(v)=0,~~\hbox{ for all $0\le v\le\theta$.}
\end{equation}
For such nonlinearities, there is no hair trigger effect even if the initial conditions
were to obey~\eqref{26aug1314}: the fuel temperature has to at least
cross the ignition threshold $\theta$ for the reaction to start, see~\cite{Kanel,Zlatos06}
for a detailed analysis of this phenomenon for a single equation.  The analysis in~\cite{LaRoqRyz}
completely fails  in this situation. 
 
Let us now briefly explain the new ingredients in the proof of Theorem~\ref{thm-26aug902}.  
The main role of the hair trigger effect in~\cite{LaRoqRyz} was to ensure that when reaction starts 
in a given region, the concentration is depleted there very quickly. In the present setting,
this step is accomplished by Lemma~\ref{lem:hot-zones-u-lossbis} introduced in Section~\ref{sec:forward}
below, which replaces the hair trigger effect.  In particular, the decay of the fuel concentration here
requires a large average of the fuel temperature, in contrast to~\cite{LaRoqRyz}. 
The first really new ingredient is 
Lemma~\ref{lem:force-gen-boundbis} in Section~\ref{sec:modMP} that, in a sense, provides a 
non-local version of the Martin-Pierre trick that works
in the case $\kappa\ge \nu$ and bounds an average of the
forcing in the temperature equation by an average of the fuel concentration. The second component 
of the proof that is novel, to the best of our knowledge, is a remarkable cancellation in the Duhamel expansion that is 
repeatedly used in Section~\ref{sec:duhamel}, using a hot-cold decomposition at discrete time steps of a fixed length. 
Without this cancellation, the proof would not have worked. 

This paper is organized as follows. Section~\ref{sec:averages} 
describes the non-local version of the Martin-Pierre trick. We also recall there
why bounding the local averages of the solution is sufficient to obtain a global-in-time
uniform bound. Section~\ref{sec:proof-main-lemma} contains the proof of Lemma~\ref{lem-26aug902},
which is the main step in the proof of Theorem~\ref{thm-26aug902}.  Finally, the
forward-in-time decay Lemma~\ref{lem:hot-zones-u-lossbis} is proved in Section~\ref{sec:aver-lem-proof}.

  
{\bf Acknowledgment.} This work was supported by NSF grants DMS-2205497 and DMS-2510166. 
The contribution of ChatGPT to the proofs was indispensable and most certainly
sufficient to include it as 
a co-author if it were a human. The text of the paper was written without
any AI assistance except for checking for misprints.

\section{A modified Martin-Pierre trick and heat semi-group averages} \label{sec:averages}

\subsection{The original Martin-Pierre trick}\label{sec:martin-pierre}

Let us first recall the beautiful and short trick for proving the boundedness of the solutions to the thermo-diffusive system~\eqref{main-PDE}
when $\nu\ge\kappa$ (the fuel diffuses faster than temperature) from~\cite{MartinPierre}. Let us denote by
\begin{equation}\label{26aug1123}
q(t,x)=u(t,x)g(v(t,x))
\end{equation}
the forcing term in~\eqref{main-PDE} and use the Duhamel formula to represent $u(t,x)$ and $v(t,x)$ as 
\begin{equation}\label{26aug1120}
u(t,x)=[\cS_\nu(t)u_0](x)-\int_0^t [\cS_\nu(t-s)q(s,\cdot)](x)ds,
\end{equation}
and
\begin{equation}\label{26aug1121}
v(t,x)=[\cS_\kappa(t)v_0](x)+\int_0^t [\cS_\kappa(t-s)q(s,\cdot)](x)ds.
\end{equation}
Here and in the rest of the paper, we denote by $\cS_a(t)$ the semi-group associated with the heat equation with a diffusivity $a>0$ on $\Rm^d$:
\begin{equation}\label{26aug1141}
[\cS_a(t)f](x)=\int G_a(t,x-y)f(y)dy,~~G_a(t,x)=\farc{1}{(4\pi at)^{d/2}}e^{-|x|^2/(4at)}.
\end{equation}
A simple but very useful observation is that  
\begin{equation}
a^{d/2}G_a(t,x)\ge b^{d/2}G_b(t,x),~~\hbox{ for all $t>0$ and $x\in\Rm^d$ if $a>b$.}
\end{equation}
In terms of the heat semi-group this translates into the inequality
\begin{equation}\label{26aug1126}
[\cS_b(t)f](x)\le \Big(\farc{a}{b}\Big)^{d/2}[\cS_a(t)f](x),~~\hbox{ for all $t>0$ and $x\in\Rm^d$ if $a>b$,}
\end{equation}
which holds for all non-negative functions $f(x)$.

As the function $q(t,x)$ defined in~\eqref{26aug1123} is non-negative, it follows that if $\nu\ge\kappa$ then we have, using \eqref{26aug1120} and~\eqref{26aug1121}
\begin{equation}\label{26aug1125}
v(t,x)=[\cS_\kappa(t)v_0](x)+\int_0^t [\cS_\kappa(t-s)q(s,\cdot)](x)ds\le K+\farc{\nu^{d/2}}{\kappa^{d/2}}\int_0^t [\cS_\nu(t-s)q(s,\cdot)](x)ds.
\end{equation}
Here, $K$ is the upper bound on the initial conditions in~\eqref{main-PDE-initial-conditions}. 
On the other hand, as $u(t,x)$ is non-negative, we see from \eqref{26aug1120} 
that
\begin{equation}\label{26aug1124}
\int_0^t [\cS_\nu(t-s)q(s,\cdot)](x)ds\le [\cS_\nu(t)u_0](x)\le K.
\end{equation}
We deduce immediately from \eqref{26aug1125} and~\eqref{26aug1124} that $v(t,x)$ obeys a uniform upper bound
\begin{equation}
v(t,x)\le\Big(1+\farc{\nu^{d/2}}{\kappa^{d/2}}\Big)K.
\end{equation}

\subsection{A modified Martin-Pierre trick for the averages}\label{sec:modMP}

Unfortunately, the argument in Section~\ref{sec:martin-pierre} does not apply 
in the case $\nu\le\kappa$ as the inequality~\eqref{26aug1126} does not 
bound the semi-group $\cS_\kappa(t)$ from above in terms of $\cS_\nu(t)$ in that case. 
This can be flipped by introducing an additional
smoothing by the heat semi-group that allows one once again to compare the heat kernels, as follows.  
Let us fix some $\tau>0$ and use the Duhamel formula on a time interval $[t-\tau,t]$, for 
any~$t>\tau$:
\begin{equation}\label{26aug1139}
v(t,x)=[\cS_\kappa(\tau)v(t-\tau,\cdot)](x)+\widetilde W(t,t-\tau,x). 
\end{equation}
Here, we have set 
\begin{equation}\label{26aug1131}
\widetilde W(t_1,t_2,x) = \int_{t_2}^{t_1}[\cS_\kappa(t_1-s)q(s,\cdot)](x)ds,
\end{equation}
with $q(t,x)$ given by~\eqref{26aug1123}. We have the following smoothed version of the Martin-Pierre bound on the forcing term in terms of the fuel concentration. 
\begin{lemma}\label{lem:force-gen-boundbis}
For any $\tau>0$ we have 
\begin{equation}\label{26aug920}
[\cS_\kappa(\tau)\widetilde W(t,t-\tau,\cdot)](x)\le M_{\kappa,\nu}[\cS_\kappa(2\tau)u(t-\tau,\cdot)](x),~~\hbox{for all $t\ge\tau$ and $x\in\Rm^d$,}
\end{equation}
with
\begin{equation}\label{26aug1202}
M_{\kappa,\nu}=\Big(\farc{2\kappa-\nu}{\kappa}\Big)^{d/2}.
\end{equation}
\end{lemma}
{\bf Proof.} 
 Applying the semi-group $\cS_\kappa(\tau)$ to both sides of~\eqref{26aug1131} with $t_1=t$ and $t_2=t-\tau$ gives 
\begin{equation}\label{26aug1112}
[\cS_\kappa(\tau)\widetilde W(t,t-\tau,\cdot)](x)= \int_{t-\tau}^{t}[\cS_\kappa(\tau+t-s)q(s,\cdot)](x)ds.
\end{equation}
Next, for each $s\in[t-\tau,t]$, set
\[
a_s:=\kappa(t+\tau-s)-\nu(t-s)
=\kappa\tau+(\kappa-\nu)(t-s).
\]
This allows us to represent the semi-group on the right side of~\eqref{26aug1112} as 
\begin{equation} \label{shifted-S-k-to-S-nu}
\cS_\kappa(t+\tau-s)=e^{\kappa(t+\tau-s)\Delta}=e^{a_s\Delta}e^{\nu(t-s)\Delta}=
\cS_{1}(a_s)\cS_\nu(t-s),
\end{equation}
so that
\begin{equation}\label{26aug1133}
[\cS_\kappa(\tau)\widetilde W(t,t-\tau,\cdot)](x)= \int_{t-\tau}^{t}[\cS_{1}(a_s)\cS_\nu(t-s)q(s,\cdot)](x)ds.
\end{equation}
Furthermore, since $0<\nu\le\kappa$ and $t-\tau<s<t$, we have 
\begin{equation}\label{a-s-param-bounds}
\kappa\tau\le a_s\le A:=(2\kappa-\nu)\tau.
\end{equation}
Thus,~\eqref{26aug1126} implies that
\begin{equation}\label{26aug1114bis}
\cS_1(a_s)\le \Big(\farc{A}{a_s}\Big)^{d/2}\cS_1(A)\le \Big(\farc{2\kappa-\nu}{\kappa}\Big)^{d/2}\cS_1(A).
\end{equation}
Using this inequality in~\eqref{26aug1133}, we obtain
\begin{equation}\label{26aug1133bis}
[\cS_\kappa(\tau)\widetilde W(t,t-\tau,\cdot)](x)\le \Big(\farc{2\kappa-\nu}{\kappa}\Big)^{d/2}\cS_1(A)\int_{t-\tau}^{t}[\cS_\nu(t-s)q(s,\cdot)](x)ds.
\end{equation}

On the other hand, as in~\eqref{26aug1120}, the Duhamel formula for the fuel concentration $u(t,x)$ on 
the time interval~$[t-\tau,t]$ gives
\begin{equation}
u(t,x)=
[\cS_\nu(\tau)u(t-\tau,\cdot)](x)
-\int_{t-\tau}^{t}[\cS_\nu(t-s)q(s,\cdot)](x)\,ds.
\end{equation}
Since $u(t,x)$ is nonnegative, it follows that
\begin{equation}  \label{force-gain-small}
\int_{t-\tau}^{t}[\cS_\nu(t-s)q(s,\cdot)](x)\,ds
\le [\cS_\nu(\tau)u(t-\tau,\cdot)](x).
\end{equation}
Combining this bound with~\eqref{26aug1133bis} leads to 
\begin{equation}\label{26aug1134}
\bal
[\cS_\kappa(\tau)&\widetilde W(t,t-\tau,\cdot)](x)\le  \Big(\farc{2\kappa-\nu}{\kappa}\Big)^{d/2}[\cS_1(A) \cS_\nu(\tau)u(t-\tau,\cdot)](x)\\
&=
\Big(\farc{2\kappa-\nu}{\kappa}\Big)^{d/2}[\cS_1(A+\nu\tau)u(t-\tau,\cdot)](x)\\
&=\Big(\farc{2\kappa-\nu}{\kappa}\Big)^{d/2}[\cS_1(2\kappa\tau)u(t-\tau,\cdot)](x)=\Big(\farc{2\kappa-\nu}{\kappa}\Big)^{d/2}[\cS_\kappa(2\tau)u(t-\tau,\cdot)](x).
\enbal
\end{equation}
We used the definition \eqref{a-s-param-bounds} of $A$ in the next-to-last identity above. This gives~\eqref{26aug920}.~$\Box$  

\subsection{Bounds on the averages are sufficient} 

Lemma~\ref{lem:force-gen-boundbis} in itself is not as powerful as the original Martin-Pierre trick 
for two reasons.
First, it does not directly bound the forcing term $\widetilde W(t,t-\tau,x)$ in the Duhamel formula~\eqref{26aug1139} but only its average
by the heat semi-group.  Second, it gives no control of the ``initial condition'' term $\cS_\kappa(\tau)v(t-\tau,\cdot)$ since we have no a priori bound on $v(t-\tau,x)$.  

To deal with the former issue, we
recall the following result
of~\cite{HLV,LaRoqRyz}. For any $r>0$, $x\in\Rm^d$, and $t>r^2$ define the parabolic cylinder
\begin{equation}
Q_{r}^-(t,x):=(t-r^2,t)\times B_{r}(x).
\end{equation}
The parabolic BMO semi-norm is defined~by 
\begin{equation*}
\|f \|_{\pbmo} = \sup \Big \{ \frac{1}{|Q|} \int_Q | f - (f)_Q | dxdt  \Big\},
\end{equation*}
where the supremum is taken over all parabolic 
cylinders $Q \subset \Rm_+\times\mathbb{R}^n$.  
The main result of \cite{HLV} and~\cite{LaRoqRyz} is the parabolic BMO bound
on the function $v(t,x)$. 
\begin{proposition}\label{lem-jul2202}
Let $u(t,x)$, $v(t,x)$ be the solution to the system \eqref{main-PDE}  
with the initial conditions~$u(0,x)=u_0(x)$ and~$v(0,x)=v_0(x)$ 
that satisfy~(\ref{main-PDE-initial-conditions}). Assume, in addition, that $g(v)\ge 0$ for
all $v\ge 0$. 
There exists a constant $C_1$ that depends on the constant $K$   in (\ref{main-PDE-initial-conditions}), 
and also on the dimension $n$ and the diffusivities $\kappa>0$ and $\nu>0$ but not on the function
$g(v)$ such that 
\begin{equation}\label{jul2302}
\|v\|_{\pbmo(\Rm_+\times\Rm^d)}\le C_1.
\end{equation}
\end{proposition}
The above bound does not depend on the nature of the nonlinearity $g(v)$. It says, essentially,
that if two functions $p$ and $r$ satisfy a pair of forced heat equations
\begin{equation}
\bal
&p_t=\kappa\Delta p+q(t,x),~~t>0,~x\in\Rm^d,\\
&r_t=\nu\Delta r+q(t,x),~~t>0,~x\in\Rm^d,
\\
\enbal
\end{equation}
with the same force $q(t,x)$, and $r(t,x)$ is uniformly bounded in $L^\infty(\Rm_+\times\Rm^d)$,
then $p(t,x)$ is bounded in the parabolic BMO norm. This comes from the fact that $p(t,x)$ 
and $r(t,x)$ are essentially related by a space-time singular integral operator. The details
of this argument can be found in~\cite{LaRoqRyz}. 

Let us also recall the John-Nirenberg inequality.
\begin{proposition}\label{lem-john-nir} (John--Nirenberg inequality) 
There exist constants $A, B>0$, such that for all parabolic cylinders~$Q=Q_R(t_0,x_0)$ and all~$\lambda>0$, we have
\begin{equation}\label{jul2536}
\big|\{ (t,x) \in Q:~|p(t,x) - (p)_Q | \ge \lambda \} | \le B \exp \Big( -\frac{A \lambda}{\|p\|_{\pbmo} }\Big) |Q|,
\end{equation}
for all $p\in\pbmo$. Here, $|Q|$ is the $(n+1)$-dimensional Lebesgue measure of $Q$.
\end{proposition}

As a consequence of the parabolic BMO bound~\eqref{jul2302} and the John-Nirenberg inequality, it was 
shown in~\cite{LaRoqRyz}, under the present assumptions on $g(v)$, that, 
in order to prove the uniform bound~\eqref{26jul2604} on $v(t,x)$ it suffices to bound the
averages of $v(t,x)$ over parabolic cylinders. More precisely, we have the following. 

\begin{proposition}[Corollary 3.6 of~\cite{LaRoqRyz}]\label{prop-26jul2602}
For any $\eps>0$ and $R>0$ there exists $C_\eps$ such that for any~$t_0\ge 2R^2$ 
and $x_0\in\Rm^d$ we have 
\begin{equation}\label{26jul2620}
\sup_{Q^{{ {-}}}_R(t_0,x_0)}v(t,x)\le C_\eps(v)_{Q_{2R}^-(t_0,x_0)}^{1-\eps}
\exp\big(C_{\eps} (v)_{Q_{2R}^-(t_0,x_0)} \big). 
\end{equation}
\end{proposition}
Thus, to obtain a uniform bound on $v(t,x)$ it suffices to bound its averages over parabolic cylinders. 


It is straightforward to replace the parabolic cylinder averages by the heat semigroup averages that appear
in Lemma~\ref{lem:force-gen-boundbis}. 
Note that for any $\tau>0$ fixed, there exists $m_\tau>0$ so that the heat kernel $G_\kappa(t,x)$ defined by~\eqref{26aug1141}
satisfies a lower bound 
\begin{equation}
G_\kappa(\tau,x)\ge m_\tau,~~\hbox{ for all $x\in B_1(0)$.}
\end{equation}
It follows that for any non-negative function $p(x)$ we have
\begin{equation}
[\cS_\kappa(\tau)p](x)=\int G_\kappa(\tau,x-y)p(y)dy\ge m_\tau\int_{B_1(x)}p(y)dy.
\end{equation}
As a consequence, given a function $z(t,x)$, for any $\tau>0$ we have an upper bound
\begin{equation}
\bal
|B_1(0)|(z)_{Q_1^-(t, x)} =\int_{t-1}^t\int_{B_1(x)}z(s,y)dy\le \farc{1}{m_\tau}\int_{t-1}^t [\cS_\kappa(\tau)z(s,\cdot)](x)ds\le 
\farc{1}{m_\tau}\sup_{s\ge 0}\|\cS_\kappa(\tau)z(s,\cdot)\|_{L^\infty(\Rm^d)}.
\enbal
\end{equation}
Thus, the conclusion of Theorem~\ref{thm-26aug902} would follow from Proposition~\ref{prop-26jul2602}
and the following lemma.
\begin{lemma}\label{lem-26aug902}
There exist $\tau>0$ and $C_\tau>0$ so that  
\begin{equation}\label{26aug912}
[\cS_\kappa(\tau)v(t,\cdot)](x)\le C_\tau,~~\hbox{ for all $t>0$ and $x\in\Rm^d$.}
\end{equation}
\end{lemma}
The rest of the paper contains the proof of Lemma~\ref{lem-26aug902}.
Before proceeding with that proof, let us make the following comment. We can represent $v(t,x)$ as
\begin{equation}\label{26aug1142}
v(t,x)=[\cS_\kappa(t)v_0](x)+W(t,x).
\end{equation}
Here, we have set
\begin{equation}\label{26aug910}
W(t,x)=\int_0^t [\cS_\kappa(t-s)q(s,\cdot)](x)ds.
\end{equation}
The term  
\begin{equation}
\bar v(t,x)=\cS_\kappa(t)v_0(x),
\end{equation}
on the right side of~\eqref{26aug1142} is a solution to the heat equation and satisfies the trivial upper bound
\begin{equation}\label{26aug906}
\bar v(t,x)\le K,
\end{equation}
that follows from the assumption~\eqref{main-PDE-initial-conditions} on the initial condition $v_0$. 
Therefore, the conclusion of Lemma~\ref{lem-26aug902} would follow from the a priori bound
\begin{equation}\label{26aug908}
[\cS_\kappa(\tau)W(t,\cdot)](x)\le C_\tau,~~\hbox{ for all $t>0$ and $x\in\Rm^d$,}
\end{equation}
and this is what we will prove. 


\section{The proof of Lemma~\ref{lem-26aug902}} \label{sec:proof-main-lemma} 

\subsection{A forward-in-time fuel decay estimate}\label{sec:forward}

Lemma~\ref{lem:force-gen-boundbis} provides some sort of a bound on the forcing term in the equation for the temperature~$v(t,x)$
in terms of the fuel concentration $u(t,x)$. In order to be able to close the argument, we 
need a decay bound on~$u(t,x)$ in the regions where~$v(t,x)$ is sufficiently large. This is provided by the following lemma. 
Let us fix $\vartheta>0$ such that  
\begin{equation}\label{eq::vatheta-defbis}
   h_\vartheta:= \inf_{v > \vartheta}g(v)> 0. 
\end{equation}
\begin{lemma}[Fuel loss in a predominantly ignited cylinder]
\label{lem:hot-zones-u-lossbis} 
Let $u(t,x)$, $v(t,x)$ be non-negative  solutions to the system~(\ref{main-PDE}) with $0 < \nu<\kappa$ and $g (v)\ge 0$ satisfying (\ref{g-liminf-bound}). Then,  
for any~$\eta\in(0,1)$ and any~$\vartheta>0$ that satisfies (\ref{eq::vatheta-defbis}),  there exist $\tau_\eta > 0$ and
$\varepsilon_\eta>0$ such that if, for some $t \ge \tau_\eta$ and~$x\in\Rm^d$ we have
\begin{equation}
 \frac{\bigl|\{v<\vartheta\}\cap Q^-_{\ell_\eta}(t,x)\bigr|}
      {\tau_\eta|B_{\ell_\eta}(0)|}
 \le\varepsilon_\eta,
 \label{eq:cold-fraction}
\end{equation}
with $\ell_\eta=\sqrt{\tau_\eta}$, then 
\begin{equation}  \label{eq:fuel-loss}
 u(t,x)\le\eta [\cS_\kappa(\tau_\eta)u(t-\tau_\eta,\cdot)](x).
\end{equation}
\end{lemma}
We postpone the proof of this lemma until Section~\ref{sec:aver-lem-proof} below. As we have mentioned, together with
Lemma~\ref{lem:force-gen-boundbis}  it will allow us to close the argument.  
 
\subsection{The discrete time Duhamel steps}\label{sec:duhamel}

We now prove   Lemma~\ref{lem-26aug902} 
that, as we have mentioned, implies the conclusion of Theorem~\ref{thm-26aug902}. Let us fix $\eta\in(0,1)$, choose $\tau_\eta$ as in Lemma~\ref{lem:hot-zones-u-lossbis}
and note that it suffices to prove \eqref{26aug912} for~$t>2\tau_\eta$ since $v(t,x)$ is a priori bounded for all $0\le t\le 2\tau_\eta$ by a constant that depends on~$\tau_\eta$. 
The proof is based on the iterative Duhamel expansion, in time steps of the size $\tau_\eta$. Very roughly, the idea is that in the regions where $v(t,x)$ is large on average,
the fuel concentration will drop because of the conclusion of Lemma~\ref{lem:hot-zones-u-lossbis}. On the other hand, the fuel concentration allows us to bound heat semi-group averages
of $v(t,x)$ because of Lemma~\ref{lem:force-gen-boundbis}.

We introduce the function
\begin{equation}\label{26aug916}
V(t,x) = [\cS_\kappa(\tau_\eta)W(t-\tau_\eta,\cdot)](x),
\end{equation}
with $W(t,x)$ defined by~\eqref{26aug910} and note that the estimate \eqref{26aug908}, that we seek to prove, is equivalent to the bound
\begin{equation}\label{26aug1018}
V(t,x)\le C_\tau,~~\hbox{ for all $t>\tau_\eta$ and $x\in\Rm^d$.}
\end{equation}
The function $W(t,x)$ satisfies
\begin{equation}
W_t=\kappa\Delta W+q(t,x).
\end{equation}
Thus, for any $t>\tau_\eta$, the Duhamel formula gives 
\begin{equation}\label{26aug918}
\bal
W(t,x)&=[\cS_\kappa(\tau_\eta)W(t-\tau_\eta,\cdot)](x)+\int_{t-\tau_\eta}^t [\cS_\kappa(t-s)q(s,\cdot)](x)ds\\
&=[\cS_\kappa(\tau_\eta)W(t-\tau_\eta,\cdot)](x)+\tW(t,t-\tau_\eta,x)=V(t,x)+\tW(t,t-\tau_\eta,x).
\enbal
\end{equation}
Here, $\tW(t,t-\tau,x)$ is defined by~\eqref{26aug1131}.
As $\tW(t,t-\tau, x)$ is non-negative, we see from \eqref{26aug918} that the function $V(t,x)$ 
satisfies 
\begin{equation}\label{26aug932}
V(t,x)\le W(t,x)\le v(t,x).
\end{equation}
The second inequality above follows from~\eqref{26aug1142} and non-negativity of $\cS_\kappa(t)v_0$. In addition, 
(\ref{26aug918}) implies the following Duhamel formula for the function $V(t,x)$:
\begin{equation}\label{X-recursivebis}
\bal
V(t,x) &=[\cS_\kappa(\tau_\eta)\circ \cS_\kappa(\tau_\eta)W(t-2\tau_\eta,\cdot)](x)+[\cS_\kappa(\tau_\eta)\tW(t-\tau_\eta,t-2\tau_\eta,\cdot)](x)
\\
&= [\cS_\kappa(\tau_\eta)V(t-\tau_\eta,\cdot)](x) + [\cS_\kappa(\tau_\eta)\tW(t-\tau_\eta,t-2\tau_\eta,\cdot)](x). 
\enbal
\end{equation}
Lemma~\ref{lem:force-gen-boundbis} gives us a uniform bound on the second term above:
\begin{equation}
[\cS_\kappa(\tau_\eta)\widetilde W(t-\tau_\eta,t-2\tau_\eta,\cdot)](x)\le M_{\kappa,\nu} [\cS_\kappa(2\tau_\eta)u(t-2\tau_\eta,\cdot)](x)\le  KM_{\kappa,\nu}.
\end{equation}

Thus, our main concern is to deal with the first term on the right side of~\eqref{X-recursivebis}.  Let us recall the parabolic BMO bound \eqref{jul2302}
on the function $v(t,x)$ and the John-Nirenberg inequality~\eqref{jul2536}. We may use them to find $Z_\eta > 0$ sufficiently large, so that 
\begin{equation}\label{26aug922}
 \farc{|\{v < \vartheta\}\cap Q_{\ell_\eta}^-(t, x)| }{\tau_\eta|B_{\ell_\eta}(0)|}\le \varepsilon_\eta,~~\hbox{ for all $x\in\Rm^d$ s.t.   $(v)_{Q_{\ell_\eta}^-(t, x)} > Z_\eta$.} 
\end{equation}
For a given $t\ge 0$, we now define the ``hot'' set 
\begin{equation}
U_t = \big\{x\in\Rm^d:~ (v)_{Q_{\ell_\eta}^-(t, x)} > Z_\eta\big\}.
\end{equation}
Proposition~\ref{prop-26jul2602} allows us to find $M_\eta>0$ so that 
\begin{equation}\label{26aug933}
v(t,x)\le M_\eta,~~\hbox{ for all $t>0$ and $x\not\in U_t$.}
\end{equation}
Thus, our main task is to bound $v(t,x)$ on the hot set $U_t$. 
%
Let us observe that Lemma~\ref{lem:hot-zones-u-lossbis} and~\eqref{26aug922} 
imply that 
\begin{equation}\label{26aug923}
u(t,x)\le\eta [\cS_\kappa(\tau_\eta)u(t-\tau_\eta,\cdot)](x),~~\hbox{for all $t\ge 0$ and $x\in U_t$.}
\end{equation}
%
%

We will iterate the Duhamel formula~\eqref{X-recursivebis} for the function $V(t,x)$, backward in time in steps of the size $\tau_\eta$. 
To this end, we divide the first term $[\cS_\kappa(\tau_\eta)V(t-\tau_\eta,\cdot)](x)$ 
on the right side of~\eqref{X-recursivebis} into the hot and the cold parts.
On the hot parts we will rely on the exponential decay~\eqref{26aug923} of~$u(t,x)$, while on the
cold parts we will use the a priori bound~\eqref{26aug933}. 

To perform the hot-cold decomposition and set up the iteration, fix $t>2\tau_\eta$ and define iteratively the operators:
\begin{equation}\label{26aug926}
\cT^t_{n+1} = \cT^t_n \circ \cS_\kappa(\tau_\eta)\circ\one(x \in U_{t-(n+1)\tau_\eta}),~~\cT^t_0 = I,
\end{equation}
and
\begin{equation}\label{26aug927}
\cF^t_{n+1} = \cT^t_n \circ \cS_\kappa(\tau_\eta)\circ\one(x \not\in U_{t-(n+1)\tau_\eta}),~~\cF^t_0 = 0.
\end{equation}
These operators come from iterating the leading term in~\eqref{X-recursivebis} 
backward in time, and splitting
\begin{equation}
\cS_\kappa(\tau_\eta)=\cS_\kappa(\tau_\eta)\circ\one(x \in U_{t-(n+1)\tau_\eta})
+\cS_\kappa(\tau_\eta)\circ\one(x \not\in U_{t-(n+1)\tau_\eta})
\end{equation}
at each time step of the size $\tau_\eta$. 

As $\cS_\kappa(\tau_\eta)$ preserves positivity and 
\begin{equation}
\|\cS_\kappa(\tau_\eta)\|_{L^\infty\to L^\infty}=1,
\end{equation}
it follows that
\begin{equation}\label{26aug930}
\|\cT_n^t\|_{L^\infty\to L^\infty}\le 1,~~\|\cF_n^t\|_{L^\infty\to L^\infty}\le 1.
\end{equation}
In addition, the operators $\cT_n^t$ and $\cF_n^t$ also preserve positivity and are monotonic:  
if $p_1(x)\ge p_2(x)$ for all $x\in\Rm^d$, then 
\begin{equation}
\hbox{$\cT_n^tp_1(x)\ge \cT_n^tp_2(x)$ for all $x\in\Rm^d$ and $n\ge 1$,}
\end{equation}
and
\begin{equation}
\hbox{$\cF_n^tp_1(x)\ge \cF_n^tp_2(x)$ for all $x\in\Rm^d$ and $n\ge 1$.}
\end{equation}

To iterate the Duhamel formula~\eqref{X-recursivebis}, we note that
\begin{equation}
\cS_\kappa(\tau_\eta)=\cT_1^t+\cF_1^t,
\end{equation}
and we re-write the second line of~\eqref{X-recursivebis} as 
\begin{equation}\label{26aug924}
V(t,\cdot) =\cT^t_1V(t-\tau_\eta,\cdot) + \cF^t_1V(t-\tau_\eta,\cdot) +\cS_\kappa(\tau_\eta)\tW(t-\tau_\eta,t-2\tau_\eta,\cdot).
\end{equation}
In turn, we write (\ref{26aug924}) for $t>(n+2)\tau_\eta$ as 
\begin{equation}\label{26aug925}
\bal
V(t-n\tau_\eta,\cdot) &=\cT^{t-n\tau_\eta}_1V(t-(n+1)\tau_\eta,\cdot) + \cF^{t-n\tau_\eta}_1V(t-(n+1)\tau_\eta,\cdot) \\
&+\cS_\kappa(\tau)\tW(t-(n+1)\tau_\eta,t-(n+2)\tau_\eta,\cdot).
\enbal
\end{equation}
The iterative definition \eqref{26aug926} gives
\begin{equation}\label{26aug1002}
\bal
(\cT_n^t\circ \cT^{t-n\tau_\eta}_1)V(t-(n+1)\tau_\eta,\cdot)&=(\cT_n^t\circ \cS_\kappa(\tau_\eta)[\one(x \in U_{t-(n+1)\tau_\eta})V(t-(n+1)\tau_\eta,\cdot)]\\
&= \cT^t_{n+1}V(t-(n+1)\tau_\eta,\cdot) .
\enbal
\end{equation}
On the other hand, \eqref{26aug927} leads to
\begin{equation}\label{26aug928}
\bal
(\cT_n^t\circ \cF^{t-n\tau_\eta}_1)V(t-(n+1)\tau_\eta,\cdot)&=(\cT_n^t\circ \cS_\kappa(\tau_\eta)[\one(x \not\in U_{t-(n+1)\tau_\eta})V(t-(n+1)\tau_\eta,\cdot)]\\
&= \cF^t_{n+1}V(t-(n+1)\tau_\eta,\cdot) .
\enbal
\end{equation}
Therefore, applying $\cT_n^t$ to both sides of \eqref{26aug925} and using \eqref{26aug1002} and \eqref{26aug928}, we obtain
\begin{equation}\label{26aug929}
\bal
\cT^t_nV(t-n\tau_\eta,\cdot) &= \cT^t_{n+1}V(t-(n+1)\tau_\eta,\cdot) + \cF^t_{n+1}V(t-(n+1)\tau_\eta,\cdot) \\
&+ \cT^t_n\cS_\kappa(\tau_\eta)\tW(t-(n+1)\tau_\eta,t-(n+2)\tau_\eta).
\enbal
\end{equation}
Summing \eqref{26aug929} over $0\le n\le N-1$ gives 
\begin{equation}
\bal
\sum_{n=0}^{N-1}\cT^t_nV(t-n\tau_\eta,\cdot)&=
\sum_{n=0}^{N-1}\cT^t_{n+1}V(t-(n+1)\tau_\eta,\cdot) +\sum_{n=0}^{N-1} \cF^t_{n+1}V(t-(n+1)\tau_\eta,\cdot) \\
&+\sum_{n=0}^{N-1} \cT^t_n\cS_\kappa(\tau_\eta)\tW(t-(n+1)\tau_\eta,t-(n+2)\tau_\eta).
\enbal
\end{equation}
Recalling that $\cT_0^t=I$ and setting $t_0=t-N\tau_\eta$, we obtain 
the iterated version of the Duhamel formula~\eqref{X-recursivebis}
\begin{equation}\label{26aug931}
V(t,\cdot)= \cT^t_NV(t_0,\cdot) + \sum_{n=1}^{N}\cF^t_{n}V(t-n\tau_\eta,\cdot)+  
\sum_{n=0}^{N-1}\cT^t_n\cS_\kappa(\tau_\eta)\tW(t-(n+1)\tau_\eta,t-(n+2)\tau_\eta).
\end{equation}
We now choose $N$ so that 
\begin{equation}\label{26aug1014}
\tau_\eta\le t_0\le 2\tau_\eta.
\end{equation} 
We will bound the three
terms on the right side of~\eqref{26aug931} separately. 

Going back to \eqref{26aug930}, we see that the first term on the right side of~\eqref{26aug931} is bounded because of the choice of $N$ that gives~\eqref{26aug1014}: 
\begin{equation}
I_1:=\cT^t_NV(t_0,\cdot)\le \sup_{0\le t \le 2\tau_\eta}\|V(t,\cdot)\|_{L^\infty} \le C,
\end{equation}
with some $C$ that depends on the initial conditions. 

For the second term on the right side of (\ref{26aug931}), 
we recall that the definition~\eqref{26aug927}
of the operators $\cF_n^t$ involves the restriction of the argument
to the complement of the set $U_{t-n\tau_\eta}$ and also 
note that, because of \eqref{26aug932} and~\eqref{26aug933}, together with the monotonicity of $\cF_n^t$, we have 
\begin{equation}
\one(x \not\in U_{t-n\tau_\eta}) V(t-n\tau_\eta,x)\le M_{\eta}. 
\end{equation}
It follows that 
\begin{equation}
\cF^t_{n}V(t-n\tau,x)\le M_\eta \cF_n^t[1](x),
\end{equation}
while
\begin{equation}
\bal
\cF_n^t[1]&=\cT_{n-1}^t\circ \cS_\kappa(\tau_\eta)[\one(x\not\in U_{t-n\tau_\eta})]
 =\cT_{n-1}^t\circ \cS_\kappa(\tau_\eta)[1-\one(x\in U_{t-n\tau_\eta})]\\
 &=\cT_{n-1}^t[1]-\cT_n^t[1].
\enbal
\end{equation}
Therefore,  the second sum on the right side of~\eqref{26aug931} is  telescoping: 
\begin{equation}\label{26aug1008}
 \begin{aligned}
I_2 &:=\sum_{n=1}^{N}\cF^t_{n}V(t-n\tau_\eta,\cdot)\le M_\eta\sum_{n=1}^{N}\cF^t_{n}[1]
= M_\eta\sum_{n=1}^{N}\big(\cT^t_{n-1}(1)-\cT^t_n(1)\big) \\
&= M_\eta(\cT^t_0(1) -\cT^t_N(1))\le M_\eta.
    \end{aligned}
\end{equation}

Finally, we look at the last term on the right side of~\eqref{26aug931}. 
Using Lemma~\ref{lem:force-gen-boundbis} and the monotonicity of the operators $\cT_n^t$, we obtain 
\begin{equation}\label{26aug1006}
    \begin{aligned}
I_3 &:=\sum_{n=0}^{N-1}\cT^t_n\cS_\kappa(\tau_\eta)\tW(t-(n+1)\tau_\eta,t-(n+2)\tau_\eta)
\\
&\le M_{\kappa,\nu}\sum_{n=0}^{N-1}(\cT^t_n\cS_\kappa(\tau_\eta))\cS_\kappa(\tau_\eta)u(t-(n+2)\tau_\eta).
\end{aligned}
\end{equation}
The definitions~\eqref{26aug926} and~\eqref{26aug927} imply that
\begin{equation}\label{26aug1010}
\cT_n^t\cS_\kappa(\tau_\eta)=\cT_{n+1}^t+\cF_{n+1}^t,
\end{equation}
so that
\begin{equation}
\cT_n^t\cS_\kappa(\tau_\eta)\cS_\kappa(\tau_\eta)=[\cT_{n+1}^t+\cF_{n+1}^t]\cS_\kappa(\tau_\eta)=\cT_{n+2}^t+\cF_{n+2}^t+\cF_{n+1}^t\cS_\kappa(\tau_\eta).
\end{equation}
Using this identity in~\eqref{26aug1006} gives 
\begin{equation}
\begin{aligned}
I_3&\le M_{\kappa,\nu} \sum_{n=0}^{N-1}\Big(\cT^t_{n+2}u(t-(n+2)\tau_\eta) + \cF^t_{n+2}u(t-(n+2)\tau_\eta) + \cF^t_{n+1}\cS_\kappa(\tau_\eta)u(t-(n+2)\tau_\eta)\Big)\\
        &= M_{\kappa,\nu}(B_1+B_2+B_3).
    \end{aligned}
\end{equation}
As $u(t,x)$ is uniformly bounded and the operators $\cF_n^t$ are monotonic, the last two terms can be bounded as in~\eqref{26aug1008}:
\begin{equation}
B_2+B_3\le 2  K\sum_{n=0}^{N+1}\cF^t_{n}[1]\le 2M_\eta K.
\end{equation}
As for $B_1$, we can use Lemma~\ref{lem:hot-zones-u-lossbis} since the definition~\eqref{26aug926} of the operator $\cT_n^t$ 
involves the restriction to the hot set $U_{t-n\tau_\eta}$.
This gives 
\begin{equation}
\begin{aligned}
    B_1 &= \sum_{n=2}^{N+1}\cT^t_{n}u(t-n\tau_\eta) \le \eta \sum_{n=2}^{N+1}\cT^t_{n}\cS_\kappa(\tau)u(t-(n+1)\tau_\eta) \\&
    = \eta\sum_{n=2}^{N+1}\cT^t_{n+1}u(t-(n+1)\tau_\eta) + \eta \sum_{n=2}^{N+1}\cF^t_{n+1}u(t-(n+1)\tau_\eta)
    \\& \le \eta B_1 + \cT^t_{N+2}u(t-(N+2)\tau_\eta)+\eta K\sum_{n=2}^{N+1}\cF^t_{n+1}[1]\le \eta B_1+(1+\eta)K.
\end{aligned}
\end{equation}
We used the identity~\eqref{26aug1010} above, as well as \eqref{26aug1008} once again. As $\eta\in(0,1)$, 
we see that $B_1$ is also uniformly bounded. This finishes the proof of~\eqref{26aug1018} and thus of Lemma~\ref{lem-26aug902} as well.~$\Box$ 

In order to finish the proof of Theorem~\ref{thm-26aug902}, it remains to prove Lemma~\ref{lem:hot-zones-u-lossbis}.

\section{The proof of Lemma~\ref{lem:hot-zones-u-lossbis}}\label{sec:aver-lem-proof}

We fix $\tau>0$ and $\eps>0$ to be chosen later and assume that $t>0$ and $x\in\Rm^d$ satisfy  
the assumption~\eqref{eq:cold-fraction}:  
\begin{equation}\label{26aug1021}
 \frac{\bigl|\{v<\vartheta\}\cap Q^-_{\ell}(t,x)\bigr|}
      {\tau|B_{\ell}(0)|}
\le\varepsilon.
\end{equation}
with $\ell=\sqrt{\tau}$. 

Let
$\Gamma(t,x;s,y)$ be the fundamental solution to
\begin{equation}
\partial_t w-\nu\Delta w+g(v)w=0,~~x\in\Rm^d.
\end{equation}
Then, $u(t,x)$ satisfies
\[
u(t,x)=\int_{\Rm^d}\Gamma(t,x;t-\tau,y)u(t-\tau,y)\,dy.
\]
The claim~\eqref{eq:fuel-loss} of Lemma~\ref{lem:hot-zones-u-lossbis} will follow if we find $\tau_\eta>0$ and $\eps_\eta>0$ such that 
\begin{equation}  \label{eq:fund-sol-bound} 
\Gamma(t,x;t-\tau,y)\le
    \eta G_\kappa(\tau,x-y),~~\hbox{ for all $y\in\Rm^d$},
 \end{equation}
for all $x\in\Rm^d$ that satisfy the assumption~\eqref{26aug1021} with $\eps=\eps_\eta$, $\tau=\tau_\eta$ and $\ell=\sqrt{\tau_\eta}$.   

We now prove~\eqref{eq:fund-sol-bound}. 
Since $g(v)\ge0$, the comparison to the heat equation yields
\begin{equation} \label{Gamma-bound}
0\le \Gamma(t,x;s,y)\le G_\nu(t-s,x-y).
\end{equation}
Choose now $L_\eta>1$ so large that
\begin{equation}
    \left(\frac{\kappa}{\nu}\right)^{d/2}
    \exp\!\left(
      -\frac{\kappa-\nu}{4\kappa\nu}L_\eta^2
    \right)
    \le \eta.
    \label{L-condition}
\end{equation}
 First, note that if $|x-y|\ge L_\eta\sqrt\tau$, then (\ref{Gamma-bound}) and the identity
\begin{equation}
\bal
\frac{G_\nu(\tau,x-y)}{G_\kappa(\tau,x-y)}=\Big(\farc{\kappa}{\nu}\Big)^{d/2}
\exp\Big[-|x-y|^2\Big(\farc{1}{4\nu\tau}-\farc{1}{4\kappa\tau}\Big)\Big]
 =\Big(\frac{\kappa}{\nu}\Big)^{d/2}
    \exp\Big[
      -\frac{\kappa-\nu}{4\kappa\nu}
       \frac{|x-y|^2}{\tau}\Big]
\enbal
\end{equation}
imply (\ref{eq:fund-sol-bound}). 

It remains to consider $y\in\Rm^d$ such that 
\begin{equation} \label{x-y-small}
    |x-y|<L_\eta\sqrt\tau=L_\eta\ell,
\end{equation}
with
\begin{equation}\label{26aug1038}
\ell=\sqrt{\tau}.
\end{equation}
It is here that we will need the assumption~\eqref{26aug1021}, with an appropriate $\eps_\eta>0$ and $\tau_\eta>0$. 
Duhamel's formula for the fundamental solution together with the upper bound~\eqref{Gamma-bound} gives
\begin{equation}\label{26aug1023}
\bal
 G_\nu(\tau,x-y)-\Gamma(t,x;t-\tau,y)
&=\int_{t-\tau}^{t}\!\int_{\Rm^d}G_\nu(t-s,x-z)g(v(s,z))
\Gamma(s,z;t-\tau,y)\,dz\,ds\\
 &\ge   \int_{t-\tau}^{t}\!\int_{\Rm^d}\Gamma(t,x;s,z)g(v(s,z))\Gamma(s,z;t-\tau,y)
\\
& \ge h_\vartheta  \int_{t-\tau}^{t}\int_{\{v(s,z)\ge \vartheta\}}\Gamma(t,x;s,z) \Gamma(s,z;t-\tau,y).
\enbal
\end{equation}
We used the assumption~\eqref{eq::vatheta-defbis} on the nonlinearity $g(v)$ in the last inequality above. 
We fix~$\delta\in(0,1/4)$ and further restrict the domain of integration on the right side of~\eqref{26aug1023} to 
\begin{equation}\label{26aug1022}
\bal
 G_\nu(\tau,x-y)-\Gamma(t,x;t-\tau,y)
& \ge h_\vartheta  \int_{t-2\delta\tau}^{t-\delta\tau}\int_{B_{\ell}(x)\cap\{v(s,z)\ge \vartheta\}}\Gamma(t,x;s,z) \Gamma(s,z;t-\tau,y).
\enbal
\end{equation}
 
To bound the right side of~\eqref{26aug1022}, note that
the semi-group property of the fundamental solution $\Gamma$ gives, for every $s\in(t-\tau,t)$,
\[
    \int_{\Rm^d}
      \Gamma(t,x;s,z)
      \Gamma(s,z;t-\tau,y)\,dz
    =
    \Gamma(t,x;t-\tau,y).
\]
Using this identity in~\eqref{26aug1022} gives
\begin{equation}
\begin{aligned}
G_\nu(\tau,x-y)-\Gamma(t,x;t-\tau,y)\ge
h_\vartheta \delta\tau\,\Gamma(t,x;t-\tau,y)-h_\vartheta I_{\mathrm{cold}}-h_\vartheta I_{\mathrm{out}}.
\end{aligned}
    \label{eq-G-Gamma-dif}
\end{equation}
Here, we have set 
\begin{equation}\label{26aug1029}
 I_{\mathrm{cold}}=
    \int_{t-2\delta\tau}^{t-\delta\tau}\int_{B_\ell(x)\cap\{v(s,\cdot)<\vartheta\}}
      \Gamma(t,x;s,z)
      \Gamma(s,z;t-\tau,y)\,dz\,ds,
\end{equation}
and
\begin{equation}\label{26aug1036}
    I_{\mathrm{out}}=
    \int_{t-2\delta\tau}^{t-\delta\tau}\int_{B_\ell^c(x)}
      \Gamma(t,x;s,z)
      \Gamma(s,z;t-\tau,y)\,dz\,ds.
\end{equation}
Thus, we have
\begin{equation}\label{26aug1024}
\Gamma(t,x;t-\tau,y)\le \farc{1}{1+h_\vartheta \delta\tau}\Big[G_\nu(\tau,x-y)+h_\vartheta I_{\mathrm{cold}}+h_\vartheta I_{\mathrm{out}}\Big].
\end{equation}
Note that we can make the pre-factor on the right side of \eqref{26aug1024} smaller than $\eta$ by taking $\tau>0$ sufficiently large. 
Moreover, as $\nu\le\kappa$, we have 
\begin{equation}\label{26aug1210}
G_\nu(t,x)\le \farc{\kappa^{d/2}}{\nu^{d/2}}G_\kappa(t,x).
\end{equation}
To obtain~\eqref{eq:fund-sol-bound}
we thus need to bound the terms $I_{\mathrm{cold}}$ and $I_{\mathrm{out}}$ by multiples of~$G_\nu(\tau,x-y)$. 
The integrand in their definitions can be bounded using (\ref{Gamma-bound}) as
\begin{equation}\label{26aug1035}
\Gamma(t,x;s,z)\Gamma(s,z;t-\tau,y)\le G_\nu(t-s,x-z)G_\nu(s-t+\tau,z-y).
\end{equation}
Observe that
\begin{equation}\label{26aug1027}
\bal
&G_\nu(t-s,x-z)G_\nu(s-t+\tau,z-y)=\farc{e^{-|x-z|^2/(4\nu(t-s))}e^{-|z-y|^2/(4\nu(s-t+\tau))}}{(4\pi\nu)^d(t-s)^{d/2}(s-t+\tau)^{d/2}} \\
&=G_\nu(\tau,x-y)  \farc{\tau^{d/2}e^{|x-y|^2/(4\nu\tau)}e^{-|x-z|^2/(4\nu(t-s))}e^{-|z-y|^2/(4\nu(s-t+\tau))}}{(4\pi\nu)^{d/2}(t-s)^{d/2}(s-t+\tau)^{d/2}}\\
&=G_\nu(\tau,x-y) \farc{1}{(4\pi a_s)^{d/2}}e^{|x-y|^2/(4\nu\tau)}e^{-|x-z|^2/(4\nu(t-s))}e^{-|z-y|^2/(4\nu(s-t+\tau))}.
\enbal
\end{equation}
Here, we have set 
\begin{equation}
a_s=\nu\frac{(t-s)(s-t+\tau)}{\tau}.
\end{equation}
As $t-2\delta\tau<s<t-\delta\tau$ and $\delta\in(0,1/4)$, this constant satisfies
\begin{equation}  \label{a-s-bounds}
\frac{1}{2}\nu\delta\tau\le a_s\le 2\nu\delta\tau.
\end{equation}
The exponent in~\eqref{26aug1027} can be written as 
\begin{equation}\label{26aug1028}
\bal
&E:=\frac{|x-y|^2}{4\nu\tau}-\farc{|x-z|^2}{4\nu(t-s)}-\farc{|z-y|^2}{4\nu(s-t+\tau)}\\
&=\farc{1}{4\nu\tau(t-s)(s-t+\tau)}\Big[(t-s)(s-t+\tau)|x-y|^2-\tau(s-t+\tau)|x-z|^2-\tau(t-s)|z-y|^2\Big],
\enbal
\end{equation}
while
\begin{equation}
\bal
&(t-s)(s-t+\tau)|x-y|^2-\tau(s-t+\tau)|x-z|^2-\tau(t-s)|z-y|^2\\
&=\tau(t-s)\big[|x-y|^2+|x-z|^2-|z-y|^2\big]-(t-s)^2|x-y|^2-\tau^2|x-z|^2\\
&=\tau(t-s)\big[2x^2-2x\cdot y-2x\cdot z+2z\cdot y\big]-(t-s)^2|x-y|^2-\tau^2|x-z|^2\\
&=2\tau(t-s)(x-y)\cdot(x-z)-(t-s)^2|x-y|^2-\tau^2|x-z|^2\\
&=-\big|\tau(x-z)-(t-s)(x-y)\big|^2.
\enbal
\end{equation}
Using this identity in~\eqref{26aug1028} gives 
\begin{equation}\label{26aug1032}
\bal
E=-\farc{1}{4 a_s\tau^2}\big|\tau(x-z)-(t-s)(x-y)\big|^2=-\farc{1}{4 a_s}|z-m_s|^2,
\enbal
\end{equation}
with
\begin{equation}
m_s=x+\frac{t-s}{\tau}(y-x).
\end{equation}    
Inserting the expression \eqref{26aug1032} for $E$ into \eqref{26aug1027}, we obtain 
\begin{equation}\label{26aug1031}
\bal
&G_\nu(t-s,x-z)G_\nu(s-t+\tau,z-y)=G_\nu(\tau,x-y) \farc{1}{(4\pi a_s)^{d/2}} e^{-|z-m_s|^2/(4a_s)}.
\enbal
\end{equation}
%
%
%
Going back to (\ref{26aug1029}) and using \eqref{26aug1035} and \eqref{26aug1031} leads to 
\begin{equation}\label{26aug1033}
\bal
 I_{\mathrm{cold}}&=
    \int_{t-2\delta\tau}^{t-\delta\tau}\int_{B_\ell(x)\cap\{v(s,\cdot)<\vartheta\}}
      \Gamma(t,x;s,z)
      \Gamma(s,z;t-\tau,y)\,dz\,ds\\
      &\le    \int_{t-2\delta\tau}^{t-\delta\tau}\int_{B_\ell(x)\cap\{v(s,\cdot)<\vartheta\}}G_\nu(\tau,x-y) \farc{1}{(4\pi a_s)^{d/2}} e^{-|z-m_s|^2/(4a_s)}dz ds.
\enbal
\end{equation}
Recalling the bound \eqref{a-s-bounds} on $a_s$ and, crucially, the assumption~\eqref{26aug1021} on the size of the hot set, we get
\begin{equation}\label{26aug1041}
\bal
I_{\mathrm{cold}} 
&\le   C(\delta\tau)^{-d/2}G_\nu(\tau,x-y)   \left|\{v<\vartheta\}\cap
      \bigl((t-\tau,t)\times B_{\ell}(x)\bigr)\le  C(\delta\tau)^{-d/2}\eps \tau\ell^{d}G_\nu(\tau,x-y) 
    )\right|\\
    &=C\eps\tau\delta^{-d/2}G_\nu(\tau,x-y). 
\enbal
\end{equation}
%
Here, the constant $C$ depends only on the dimension~$d$ and $\nu$.

Next, we look at the outer term in \eqref{26aug1036}. As in~\eqref{26aug1033}, 
we obtain 
\begin{equation}\label{26aug1037}
\bal
    I_{\mathrm{out}}&=
    \int_{t-2\delta\tau}^{t-\delta\tau}\int_{B_\ell^c(x)}
      \Gamma(t,x;s,z)
      \Gamma(s,z;t-\tau,y)\,dz\,ds\\
      &\le  \int_{t-2\delta\tau}^{t-\delta\tau}\int_{B_\ell^c(x)}G_\nu(\tau,x-y) \farc{1}{(4\pi a_s)^{d/2}} e^{-|z-m_s|^2/(4a_s)}dz ds.
\enbal
\end{equation}
Since  we are considering the case when $x$ and $y$ satisfy (\ref{x-y-small}), we have 
\begin{equation}\label{26aug1039}
|m_s-x|=\frac{t-s}{\tau}|y-x|\le 2\delta L_\eta\sqrt{\tau}=2\delta L_\eta\ell,~~\hbox{for all $s\in[t-2\delta\tau,t-\delta\tau]$.}
\end{equation}
Choose now $\delta>0$ sufficiently small, so that 
\begin{equation}\label{26aug1042}
2\delta L_\eta\le\farc{1}{8}.
\end{equation}
Then, (\ref{26aug1039}) guarantees that 
$m_s\in B_{\ell/2}(x)$ for every $s\in[t-2\delta\tau,t-\delta\tau]$.
In addition, \eqref{a-s-bounds} implies that 
\begin{equation}
\farc{1}{2}\nu\delta\ell^2\le a_s\le \farc{3}{2}\nu\delta\ell^2.
\end{equation}
It follows that
\begin{equation}\label{26aug1040}
\bal
&\int_{B_\ell^c(x)}\farc{1}{(4\pi a_s)^{d/2}} e^{-|z-m_s|^2/(4a_s)}dz\le \int_{B_{\ell/2}^c(m_s)}\farc{1}{(4\pi a_s)^{d/2}} e^{-|z-m_s|^2/(4a_s)}dz\\
&\le  \int_{B_{\ell/2}^c(0)}\farc{1}{(2\pi \nu\delta\ell^2)^{d/2}} e^{-|z|^2/(8\nu\delta\ell^2)}dz=
\farc{1}{(2\pi \nu)^{d/2}}\int_{B_{(2\sqrt{\delta})^{-1}}^c(0)} e^{-|y|^2/(8\nu)}dy\\
&\le C_1e^{-c_2/\delta},
\enbal
\end{equation}
with the constants $C_1$ and $c_2$ that depend only on the dimension $d$ and diffusivity $\nu$. 
Using the upper bound~\eqref{26aug1040} in~\eqref{26aug1037} gives 
\begin{equation}
    I_{\mathrm{out}}
    \le
    C\delta\tau e^{-c/\delta}
    G_\nu(\tau,x-y).
    \label{I-out-bound}
\end{equation}
Substituting (\ref{26aug1041}) and (\ref{I-out-bound}) into (\ref{eq-G-Gamma-dif}) gives
\begin{equation}\label{26aug1024bis}
\bal
\Gamma(t,x;t-\tau,y)&\le \farc{1}{1+h_\vartheta \delta\tau}\Big[G_\nu(\tau,x-y)+h_\vartheta I_{\mathrm{cold}}+h_\vartheta I_{\mathrm{out}}\Big]\\
&\le \farc{C}{h_\vartheta \delta\tau}\Big[1+h_\vartheta \eps\tau\delta^{-d/2}+h_\vartheta\delta\tau e^{-c_2/\delta}\Big]G_\nu(\tau,x-y)\\
&=C\Big[\farc{1}{h_\vartheta \delta\tau}+\eps\delta^{-d/2-1}+e^{-c_2/\delta}\Big]G_\nu(\tau,x-y)\\
&\le
\farc{C\kappa^{d/2}}{\nu^{d/2}}\Big[\farc{1}{h_\vartheta \delta\tau}+\eps\delta^{-d/2-1}+e^{-c_2/\delta}\Big]G_\kappa(\tau,x-y).
\enbal
\end{equation}
We used the inequality~\eqref{26aug1210} in the last step above. 
%
%
%

Given $\eta>0$, we now choose~$\tau_\eta>0$, $\delta_\eta\in(0,1/4)$ and $\eps_\eta>0$ as follows. 
First, we fix $L_\eta$ by~(\ref{L-condition}). Then, we choose $\delta_\eta>0$ sufficiently small  to satisfy both  \eqref{26aug1042}
and the inequality
\begin{equation}
\farc{C\kappa^{d/2}}{\nu^{d/2}}e^{-c_2/\delta_\eta}\le\farc{\eta}3.
\end{equation}
Next, we choose $\varepsilon_\eta>0$ so small that
\[
\farc{C\kappa^{d/2}}{\nu^{d/2}}\varepsilon_\eta\delta_\eta^{-n/2-1}\le\frac{\eta}{3},
\]
Finally, we choose $\tau>0$ so large that
\[
\farc{C\kappa^{d/2}}{\nu^{d/2}}\frac{1}{h_\vartheta\delta\tau}\le\frac{\eta}{3}.
\]
Then (\ref{26aug1024bis}) implies
\[
\Gamma(t,x;t-\tau,y)\le 
    \eta G_\kappa(\tau,x-y),
\]
which proves (\ref{eq:fund-sol-bound}) also for $x,y\in\Rm^d$ that satisfy (\ref{x-y-small}), finishing the proof of Lemma~\ref{lem:hot-zones-u-lossbis}.~$\Box$ 

\bibliographystyle{plain}
\bibliography{references}

\end{document}